\documentclass[11pt]{amsart}
\usepackage{amssymb,amsmath,cite,hyperref}
\usepackage[usenames]{color} \newtheorem{theorem}{Theorem}[section]
\newtheorem{proposition}[theorem]{Proposition}
\newtheorem{lemma}[theorem]{Lemma}
\newtheorem{definition}[theorem]{Definition}
\newtheorem{corollary}[theorem]{Corollary}

\newtheorem{problem}[theorem]{Problem}

\theoremstyle{definition}
\newtheorem{remark}[theorem]{Remark}

\DeclareMathOperator{\lk}{lk}
\DeclareMathOperator{\st}{st}

\newcommand{\marrow}{\marginpar{$\longleftarrow$}}
\newcommand{\eran}[1]{\textsc{\textcolor{red}{Eran says:}} \marrow \textsf{#1}}
\newcommand{\karim}[1]{\textsc{\textcolor{red}{Karim says:}} \marrow \textsf{#1}}

\title{Rigidity with few locations}
\author{Karim Adiprasito and Eran Nevo}
\address{Einstein Institute of Mathematics, The Hebrew University of Jerusalem, Jerusalem, Israel}
\email{adiprasito@math.huji.ac.il, nevo@math.huji.ac.il}
\thanks{}

\def\R{\mathbb{R}}

\begin{document}
\maketitle
\begin{abstract}
Graphs triangulating the $2$-sphere are generically rigid in $3$-space, due to Gluck-Dehn-Alexandrov-Cauchy. We show there is a \emph{finite} subset $A$ in $3$-space so that the vertices of each graph $G$ as above can be mapped into $A$ to make the resulted embedding of $G$ infinitesimally rigid. This assertion extends to the triangulations of any fixed compact connected surface, where the upper bound obtained on the size of $A$ increases with the genus. The assertion fails, namely no such finite $A$ exists, for the larger family of all graphs that are generically rigid in $3$-space and even in the plane.
\end{abstract}

\section{Introduction}
A theorem of Dehn asserts that the $1$-skeleton of every simplicial convex $3$-polytope is infinitesimally rigid \cite{Dehn}. Combined with Steinitz theorem, this gives Gluck's result that the $1$-skeleton of any simplicial $2$-sphere is \emph{generically} rigid in $\R^3$ \cite{Gluck}, i.e., the locus of realizations that are not infinitesimally rigid is of codimension one in the configuration space of all possible locations. See~\cite{Connelly:RigiditySurvey} and \cite{Pak-book} for further references and discussion.
\\
We ask the following question:
How generic does the embedding of a generically rigid graph need to be to guarantee that it is infinitesimally rigid?
\\
We give a natural precise meaning to this meta question, and partially answer it for various families of graphs, including the one mentioned above.

Let us first recall some notions pertaining to infinitesimal rigidity:
An embedding of a graph $G=(V,E)$ into $\R^d$
is any map $f: V\rightarrow \R^d$ such that $f(V)$ affinely spans $\R^d$; it defines a realization of the edges in $E$ by segments via linear extension, this realization is called the \emph{framework} $f(G)$.
A \emph{motion} of $f(G)$ is any assignment
of velocity vectors $a:V\rightarrow \R^d$ that satisfies
\begin{equation} \label {eq:infi-rigidity}
\langle a(v)-a(u),f(v)-f(u)\rangle=0
\end{equation}
 for every edge $uv \in E$. A motion $a$ is \emph{trivial} if the relation (\ref{eq:infi-rigidity}) is satisfied for {\em every pair} of vertices; otherwise $a$ is \emph{nontrivial}.
The framework $f(G)$ is \emph{infinitesimally rigid} if all its motions are trivial.
Equivalently, (\ref{eq:infi-rigidity}) says that the velocities preserve infinitesimally the distance
along an embedded edge,  and if  (\ref{eq:infi-rigidity}) applies to all pairs of
vertices then the velocities necessarily correspond to a rigid motion of the entire space.

We now arrive at the central definition of this note, quantifying the genericity of the embedding needed for an infinitesimally rigid embedding.
\begin{definition}\label{def:c}
Let $F$ be a family of graphs. We say $F$ is \emph{$d$-rigid with $c$-locations} if there exists a set $A\subseteq \mathbb{R}^d$ of cardinality $c$ such that for any graph $G=(V(G),E(G))\in F$ there exists a map $f:V(G)\rightarrow A$ such that the framework $f(G)$ is infinitesimally rigid. Denote by $c_d(F)$ the minimal such $c$.
\end{definition}
We are interested in the question whether a given infinite family of finite graphs, which is known to be generically infinitesimally $d$-rigid (namely, $d$-rigid with $\aleph_0$-locations),
is also $d$-rigid with $c$-locations, for some \emph{finite} $c$.

Clearly, the answer is yes iff \emph{any} \emph{generic} set $A$ will do in  Definition~\ref{def:c}, i.e. any $A$ where the $c\times d$ entries of its vectors are algebraically independent over the rational numbers.
Perhaps surprisingly, we show:

\begin{theorem}[Dehn-Gluck theorem with few locations]\label{thm:2-spheres}
Let $F(S^2)$ be the family of $1$-skeleta of all triangulations of the $2$-sphere. Then $F(S^2)$ is $3$-rigid with $76$-locations. Namely, $c_3(F(S^2))\leq 76$.
\end{theorem}

The phenomenon in Theorem~\ref{thm:2-spheres} generalizes to any surface, orientable or not: let $F(S)$ be the family of $1$-skeleta of all triangulations of compact connected surfaces, and let $F(g)$ be the subfamily when fixing the surface of genus $g$ (orientable or  non-orientable genus). Fogelsanger proved that any graph in $F(S)$ is generically $3$-rigid \cite{Fogelsanger}.

\begin{theorem}\label{thm:S_g}
For any $g$, $c_3(F(g))$ is finite.
\end{theorem}

We now consider the larger family of all generically $3$-rigid graphs. Some well-known open problems are to characterize this family by combinatorial means, and, concretely, whether there exists a deterministic polytime algorithm to decide if a given graph is generically $3$-rigid; see e.g. the survey \cite{Jackson-Jordan06:3Drigidity}.

Let $F_d$ be the family of all generically d-rigid finite graphs.
Note that for $d=1$, $F_1$ is the family of connected graphs, so considering a spanning tree for $G\in F_1$ shows $c_1(F_1)=2$.
Perhaps not surprisingly, $F_d$ is quantitatively more complicated for any $d\ge 2$, as shown also by Fekete and Jordan~\cite[Sec.4]{Fekete-Jordan} for $d=2$:
\begin{theorem}\label{thm:no_c(F_d)}
For any $d\ge 2$ and any finite $c$, $F_d$ is \emph{not} $d$-rigid with $c$-locations.
\end{theorem}

Let us remark that infinitesimal rigidity for (slightly) non-generic embeddings has been considered in the literature, for subfamilies of triangulated spheres and manifolds -- in the centrally symmetric case and in the balanced case, see e.g. Stanley~\cite{Stanley-CSpolytopes,Stanley-balancedCM} (phrased in the language of face rings), for convex position, see e.g. Izmestiev and Schlenker~\cite{Izmestiev-Schlenker}, and for embeddings on small grids~\cite{Fekete-Jordan} -- however, the problem of embedding with a constant number of locations, seems to be new.

\textbf{Outline}:
Preliminaries are given in Section~\ref{sec:prelim}.
We prove Theorem~\ref{thm:no_c(F_d)} in Section~\ref{sec:all}, Theorem~\ref{thm:2-spheres} in Section~\ref{sec:2-spheres}, Theorem~\ref{thm:S_g} in Section~\ref{sec:surfaces},
and conclude with open questions in Section~\ref{sec:conclude}.

\section{Preliminaries}\label{sec:prelim}

\subsection{Complexes and polytopes.}
We set some notation: the \emph{link} of a face $\sigma$ in a simplicial complex $X$ is the subcomplex $\lk_{\sigma}(X)=\{\tau\in X:\ \sigma\cap\tau=\emptyset,\ \sigma\cup\tau\in X \}$ and its (open) star is the filter $\st_{\sigma}(X)=\{\tau\in X: \sigma \subseteq \tau\}$; the set of all $l$-dimensional faces ($l$-faces for short) of $X$ is $X_l$ and its \emph{$l$-skeleton} is the subcomplex $X_{\le l}=\cup_{i\le l}X_i$.
The \emph{geometric realization} of $X$ is denoted by $|X|$; we say $X$ is a \emph{simplicial $d$-sphere / surface} if $|X|$ is a topological $d$-sphere / surface.

A $d$-dimensional polytope is \emph{simplicial} if any face in its boundary is a simplex; it is \emph{$k$-neighborly} if any subset of its vertices of size $k$ is the vertexset of some face of it; it is \emph{stacked} if it is either a $d$-simplex, or can be obtained from a stacked $d$-polytope by gluing on one of its $(d-1)$-faces a $d$-simplex.

\subsection{Rigidity}
The rigidity matrix of the framework $f(G)$, of a graph $G=(V,E)$ and an embedding $f:V\rightarrow \R^d$, is the $d|V|\times|E|$ real matrix where in the column of edge $vu\in E$ the entries in the $v$-rows are $f(v)-f(u)$, in the $u$-rows are $f(u)-f(v)$, and the other entries in this column are zero; denote this matrix by $R(f(G))$. When $|V|=n>d$, the rank of $R(f(G))$ is always
$\le dn-\binom{d+1}{2}$, and equality holds iff $f(G)$ is infinitesimally rigid.

The space of \emph{affine $2$-stresses},
or simply \emph{stresses}, of a framework $f(G)$ of a graph $G=(V,E)$ is
the real vector space
\[\{(w_{\sigma})\in \mathbb{R}^{E}: \forall \tau\in V,\ \ \sum_{\tau \in \sigma}w_{\sigma}(f(\sigma\setminus\tau)-f(\tau))=\overrightarrow{0} \}.\]
The space of affine $2$-stresses of $f(G)$ is exactly the kernel of the rigidity matrix $R(f(G))$.

For $G$ the $1$-skeleton of a simplicial $2$-sphere with $n$ vertices, it has $3n-6$ edges, and thus for $f:V\rightarrow \mathbb{R}^3$, $f(G)$ is infinitesimally rigid iff its only affine $2$-stress is the trivial (all zero) stress.

\section{All $d$-rigid graphs}\label{sec:all}
Clearly, for $F'_d \subseteq F_d$ the subfamily of \emph{minimally} generically infinitesimally $d$-rigid graphs, $c_d(F_d)=c_d(F'_d)$. For $d=2$, $F'_2$ is the family of the well studied Laman graphs.

Theorem~\ref{thm:no_c(F_d)} for $d=2$ was proved in~\cite[Sec.4]{Fekete-Jordan}, we give their short argument for completeness: Suppose by contradiction that $c=c_2(F'_2)$ is finite, and let $G$ be a Laman graph that is $2$-rigid with $c$-locations but not with $(c-1)$-locations. We may assume $|V(G)|=n>c$ (as any Laman graph is a strict subgraph of another Laman graph), so for any map $f:V(G)\rightarrow A$, $|A|=c$, $A \subset \mathbb{R}^2$, there exist two vertices $w,u\in V(G)$ with $f(w)=f(u)$.
Let $G'$ be obtained from $G$ by adding for each pair of vertices $x,y \in V(G)$  a new vertex $v=v(x,y)$ and two new edges $vx$ and $vy$.
Then $G'$ is Laman because $G'$ is obtained from $G$ by Henneberg moves, and these moves preserve the property of being Laman.
Assume by contradiction that there exists a map $f_{G'}:V(G')\rightarrow A$ with $|A|=c$ and $f_{G'}(G')$ infinitesimally rigid.
Consider $x,y \in V(G)$ with $f_{G'}(x)=f_{G'}(y)$ and the vertex $v=v(x,y)$.
Let $a(v)$ be any non-zero vector perpendicular to $f_{G'}(x)-f_{G'}(v)=f_{G'}(y)-f_{G'}(v)$, and let $a(w)=0$ for every vertex $w\neq v$. Then $a$ is a non-trivial infinitesimal motion of $G'$, which is a contradiction.


\textbf{Case $d>2$}: First proof: Argue similarly to the $d=2$ case, when adding a new vertex $v(B)$ for any $d$-subset $B\subseteq V(G)$, and connecting $v(B)$ to all vertices in $B$, to obtain $G'$ from $G\in F'_d$, where $G$ requires $c=c_d(F'_d)$ locations.

Second proof: Restrict to repeated cones over Laman graphs, forming a subfamily $F''_d$ of $F'_d$, so $c_d(F''_d)\le c_d(F'_d)\le c_d(F_d)$. By the Cone Lemma (see e.g.~\cite{TayWhiteWhiteley-Skel2Cone}), any $G\in F''_d$ is minimally generically infinitesimally $d$-rigid, and projection gives that $c_2(F'_2)\le c_d(F''_d)$, so by the $d=2$ case, $c_d(F_d)$ is infinite.
$\square$

\section{Graphs of simplicial $2$-spheres}\label{sec:2-spheres}
The proof of Theorem~\ref{thm:2-spheres} requires
the following couple of simple facts: let $G$ be a maximal planar graph on at least $5$ vertices, equivalently $G$ is the $1$-skeleton of a triangulation of the $2$-sphere different from the boundary of a tetrahedron.
\begin{lemma}\label{lem:deg3,4,5}
$G$ has a vertex of degree $\in\{3,4,5\}$.
\end{lemma}
\begin{proof}
By Euler's formula, the average degree is $<6$, and by maximality of $G$ it is $\ge 3$.
\end{proof}
\begin{lemma}\label{lem:starhole}
For any vertex $v\in G=(V,E)$, there is an edge $uv\in G$ such that the contraction of $v$ to $u$ yields a graph $G'=(V-\{v\},E')$ which is again the $1$-skeleton of a triangulation of the $2$-sphere.
\end{lemma}
The contraction as above is defined by the edgeset $E':=(E\setminus \{e\in E:\ v\in e\})\cup\{uw:\ vw\in E, w\neq u \}$.
This lemma too should be known; as we failed to find a reference we provide a proof.
\begin{proof}
Let the vertices in $\lk_v(G)$ be $u_1,u_2,\ldots,u_t$ in the cyclic order. We need to show that for some $i$, $u_iu_j\notin G$ for any $j\neq i+1,i-1 \mod{t}$, namely that $vu_i$ is not part of a missing triangle of the sphere triangulation (a.k.a. separating triangle). Then $u=u_i$ is good.

This follows from planarity, and is even simpler to argue when $\deg(v)\in\{3,4,5\}$, which suffices for our purposes:
if $\deg(v)=3$ then any $u_i$ is good. If $\deg(v)=4,5$, if $u_i$ is not good then w.l.o.g. by relabeling $(\mod t)$ $u_iu_{i+2}\in G$, so planarity shows that $u=u_{i+1}$ is good (and there exists another $u_j$ which is good as well).
\end{proof}

The proof of Theorem~\ref{thm:2-spheres} follows by showing:
\begin{theorem}\label{thm:24}
Let $A$ be a generic subset of $\R^3$, $|A|=76$. Then for any graph $G\in F(S^2)$ there exists a function $f_G:V(G)\rightarrow A$ such that

(R) the framework $f_G(G)$ is infinitesimally $3$-rigid, and

(C) for any subgraph $H\subseteq G$ of a subcomplex that triangulates a $4-$ or $5-$gon, the restriction of $f_G$ to $V(H)$ is injective.
Equivalently, for any subcomplex which is a disc consisting of up to $3$ triangles, $f_G$ is injective on its vertices.
\end{theorem}
For Theorem~\ref{thm:2-spheres} we only need (R), however, for our inductive proof to work we require (C) as well.
The theorem clearly holds if $|V(G)|\le 76$. Assume $|V(G)|>76$, let $v\in G$ be a vertex of degree $3,4$ or $5$ (it exists by Lemma~\ref{lem:deg3,4,5}), and let $uv\in G$ such that the contraction of $v$ to $u$ gives a smaller graph $G'\in F(S^2)$ ($u$ exists by Lemma~\ref{lem:starhole}).
By the induction hypothesis, there exists a function $f_{G'}$ satisfying (R) and (C) for $G'$; we will show that $f_{G'}$ can be extended to a function $f_G$ as required. We just need to show that $A$ has enough room so that $f_G(v)$ can be defined so that (R) and (C) hold for $G$.
We loose nothing by assuming that $f_{G'}$ has an image of smallest possible size.

To achieve (C), $f_G(v)$ just needs to avoid the values of $f_G$ on all the vertices (i) in the link $\lk_G(v)$, (ii) in triangles without $v$ that share an edge with triangles with $v$, and (iii) in triangles without $v$ that share an edge with triangles that share an edge with triangles with $v$.
Thus, when $\deg(v)=3$ (resp. $4$; resp. $5$), $f_G(v)$ needs to avoid a set $N$ of at most $12$ (resp. $16$; resp. $20$) values in $A$, so if $|A|>20$, we can define $f_G(v)$ so that (C) is satisfied.

We now turn to the rigidity requirement (R).
Let $c(G)$ (resp. $C(G)$) be the minimum size $c$ such that $G$ is $3$-rigid with $c$-locations (resp. and with $f_G:V(G)\rightarrow A$, $|A|=c$, satisfying property (C) as well). Thus $c(G)\leq C(G)$, hence it is enough to show that for $G,G'\in F(S^2)$ as above, $C(G)\le \max(C(G'),76)$.

Let $f_{G'}:V(G')\rightarrow A$ with $|A|=C(G')$ satisfy (R) and (C) of Theorem~\ref{thm:24}.

\textbf{First}, we notice that $f_{G'}$ has an extension ${f}_G: V(G)\rightarrow A\cup\{a\}$, with ${f}_G(v)=a$ not necessarily in $A$, such that ${f}_G(G)$ is rigid.
This follows by combining (C) with a closer look at Whiteley's proof of the Contraction Lemma~\cite{Whiteley-split}; rephrased here suitably.

\begin{proposition}[Contraction Lemma]\label{prop:WhiteleylikeContraction}
Let $vu$ be an edge in $G=(V,E)$ and assume that $u$ and $v$ have two
common neighbors $a$ and $b$. Contract $v$ to $u$ to obtain graph $G'=(V-v,E')$.
Let $f:V-v \rightarrow \mathbb{R}^3$ such that

(i) $f(G')$ is infinitesimally rigid, and

(ii) $f(a),f(b),f(u)$ are all different.\\
Then there exists an extension $f(v)$ of $f$ to $V$, such that the framework $f(G)$ is infinitesimally rigid.
\end{proposition}
To guarantee (ii) in our case, we need that $f_{G'}$ is injective on the restriction to the subgraph $H=\lk_v(G)$ of $G'$,
so (C) for $f_{G'}$ implies (ii).

\textbf{Second}, we show that in fact $f_G(v)$ can
be chosen in $A$ so that resulted $f_G$ is as required in the theorem.
We need the following crucial fact:

\begin{proposition}\label{prop:key}
Let $v\in G$ be a vertex of degree $\le b$, $A\subseteq \mathbb{R}^3$ a generic subset of size $\ge \binom{b+3}{3}$, $f: V(G)\rightarrow \mathbb{R}^3$ such that $f(V(G)\setminus\{v\})=A$ and $f(G)$ is infinitesimally rigid.

Then there exists another map $f': V(G)\rightarrow A$ that agrees with $f$ on $V(G)\setminus\{v\}$, and such that $f'(G)$ is infinitesimally rigid.
\end{proposition}
\begin{proof}
Consider the rigidity matrix
$R(f_G(G);x,y,z)$ of the framework $f_G(G)$, where $f_G(v)$ is a vector with \emph{variable} entries $(x,y,z)$. As we assumed that for some assignment $f_G(G)$ is infinitesimally rigid then, for $|V|\ge 4$, at least one of the determinants of $(3|V|-6)\times(3|V|-6)$-minors in $R(f_G(G);x,y,z)$ is nonzero; denote by $P$ such determinant. Then $P$ is a nonzero polynomial in the variables $x,y,z$ of degree at most $b$ (as $v$ is incident to at most $b$ edges).
The dimension of the space of all such polynomials is $\binom{b+3}{3}$.

Then, as $A$ is generic, no choice of different $\binom{b+3}{3}$ points in $A$ make the determinant $P$ vanish on each of the points. To see this,
put the coefficients of the polynomial $P$ into a vector $U$. Consider the $\binom{b+3}{3}$ by $\binom{b+3}{3}$ matrix $M$, where the entries of the $j$-th row are the values of our $\binom{b+3}{3}$ monomials computed using the coordinates of the $j$-th point, $(x_j,y_j,z_j)$. The determinant of $M$ is then a non-zero polynomial in $x_1,\ldots, x_{\binom{b+3}{3}}, y_1,\ldots,y_{\binom{b+3}{3}}, z_1,\ldots,z_{\binom{b+3}{3}}$ with rational coefficients.
Since our points are generic, it follows that $\det(M)\neq 0$.
On the other hand, if the value of $P$ at each of our $\binom{b+3}{3}$ points is $0$, then $MU=0$. Hence $U=0$, a contradiction, as $P$ is a nonzero polynomial.

As $|A|\ge \binom{b+3}{3}$ there is a choice $f_G(v)\in A$ as desired.
\end{proof}

Summarizing, to guarantee both (C) and (R) for $f_G(G)$, $f_G(v)$ needs to avoid at most $20+(\binom{8}{5}-1)=75$ values in $A$, which is possible for $|A|=76$.

This completes the proof that $c_3(F(S^2))\le 76$.  $\ \ \square$
\\

Variations on the above proof give stronger upper bounds for subfamilies of $F(S^2)$, which we consider next. In the rest of this section we keep the notation used in Theorem~\ref{thm:24} and its proof.

\begin{proposition}\label{prop:stacked}
The family of $1$-skeleta of \emph{stacked} $d$-polytopes is $d$-rigid with $(d+1)$-locations, but not with $d$-locations.

In particular, for the subfamily $T\subseteq F(S^2)$ of graphs of stacked $3$-polytopes, $c_3(T)=4$.
\end{proposition}
\begin{proof}[Proof sketch.]
For $v$ a vertex of degree $d$, its contraction gives the graph $G'$ of a stacked polytope. Letting $f_G(v)$ be different from the $f_G$-values of the neighbors of $v$ makes $f_G(G)$ rigid, by say the Gluing Lemma~\cite{Asimow-Roth2, Whiteley96-SomeMatroids}, or simply by inspecting the rigidity matrix of $f_G$.
\end{proof}

Note that the assertion of Proposition~\ref{prop:stacked} extends to any $d$-polytope whose $1$-skeleton contains the $1$-skeleton of a stacked $d$-polytope as a spanning subgraph. There are numerous such examples, e.g.~\cite{Shemer,Padrol}, and in particular:
\begin{corollary}\label{cor:2-nei}
The family of $1$-skeleta of \emph{2-neighborly} $d$-polytopes is $d$-rigid with $(d+1)$-locations, but not with $d$-locations.
\end{corollary}

Next we consider the subfamily $Q\subseteq F(S^2)$ of graphs that can be reduced to the complete graph $K_4$ by always contracting vertices of degree $\le 4$, obtaining a smaller triangulation of the $2$-sphere at each step.
\begin{proposition}
$c_3(Q)\le 10$.

In fact, any generic subset $A$ of $\R^3$ of size $10$ satisfies that for any graph $G\in Q$
there exists a function $f_G:\ V(G)\rightarrow A$ such that

(R') the framework $f_G(G)$ is infinitesimally $3$-rigid, and

(C') for any subgraph $H\subseteq G$ of a subcomplex that triangulates a $4$-gon, the restriction of $f_G$ to $V(H)$ is injective.
Equivalently, for any subcomplex which is a disc consisting of $2$ triangles, $f_G$ is injective on its vertices.
\end{proposition}
\begin{proof}
Contract a vertex $v\in G$ of degree $\le 4$ to vertex $u$ so that $G'$ is in $Q$.
Again, we extend $f_{G'}$ to $f_G$.
For $(C')$ to hold for $f_G$, $f_G(v)$ needs to avoid a subset $N$ of at most $8$ points in $A$.

To achieve also $(R')$, we show that for any two points $x,y\in A\setminus N$, either $f_G(v)=x$
or $f_G(v)=y$ provides the desired extension.

Note that removing an edge from $G'$ creates
a nontrivial infinitesimal motion, unique up to scaling,
say $M_{G'}:V(G')\rightarrow\mathbb{R}^3$, so for the right choice of an edge $e$ this is the nontrivial motion for the induced framework of the graph $G-v$. Then $M_{G'}$ does not preserve the distance between the vertices of $e$ up to first order. Note that the vertices of $e$ belong to $\lk_v(G)$.

Tentatively define $f_G(v)=x$; it may allow an extension of the motion $M_{G'}$ on $G'$ to $G$ for suitable $M_G(v)=v_x$. Similarly when tentatively defining $f_G(v)=y$ and $M_G(v)=v_y$. Assume by contradiction that both options extend the motion $M_{G'}$. As $M_{G'}$ is unique (up to nonzero scalar multiplication), we will get a nontrivial infinitesimal motion $M$ on the octahedron $O$ with antipodal vertices $x,y$ and equator $4$-cycle $f_G(\lk_v(G))$, by setting $M(x)=v_x, M(y)=v_y$ and $M(f_G(u))=M_{G'}(u)$ for any $u\in \lk_v(G)$.
As all $6$ vertices of $O$ are in different points of $\mathbb{R}^3$, and as $A$ is generic, by Gluck's theorem $O$ is infinitesimally rigid; a contradiction.
\end{proof}

\section{Graphs of surfaces}\label{sec:surfaces}
Barnette and Edelson~\cite{BE88,BE89} have shown that for any given compact surface $M$, the number of its \emph{irreducible} triangulations, namely those where no edge can be contracted to give a (smaller) triangulation of $M$, is finite. More strongly, and useful for our purposes, Schipper~\cite{Schipper91} has shown the following, using the Barnette-Edelson results.
\begin{lemma}(\cite[Lem.9]{Schipper91})\label{lem:Schipper}
For any compact $2$-manifold $M$ there exists a constant $n_0(M)\in \mathbb{N}$ such that any triangulation $\Delta$ of $M$, with $n>n_0(M)$ vertices, contains a vertex $v$ of degree at most $6$ such that
$v$ has a neighbor $u$ such that the contraction of $v$ to $u$ results in a (smaller) triangulation $\Delta'$ of $M$; equivalently, $vu$ is in no missing triangle of $\Delta$.
\end{lemma}

\begin{remark}~\label{rem:n_0(g)}
In fact, the argument in Schipper's proof shows more: for some $0<\epsilon <1$, independent of $M$, at least $\epsilon n$ such vertices $v$ exist. However, $\{n_0(M)\}_M$ is of course unbounded, as the minimal number of vertices needed to triangulate a surface grows with the genus.
\end{remark}

With Lemma~\ref{lem:Schipper} at hand, we can prove Theorem~\ref{thm:S_g} in the same spirit of the
proof of Theorem~\ref{thm:24}.

Let $\Delta_{\leq 1}$ denote the $1$-skeleton of a triangulation $\Delta$.
 Let $n>n_0(M)$ and contract $v$ to $u$ as in Lemma~\ref{lem:Schipper}. We need to verify that infinitesimal rigidity of the framework $f_{\Delta'_{\le 1}}(\Delta'_{\le 1})$ implies the existence of an extension $f_{\Delta_{\le 1}}(v)\in A$ that makes $f_{\Delta_{\le 1}}(\Delta_{\le 1})$ infinitesimally rigid. For this we first use Proposition~\ref{prop:WhiteleylikeContraction}; in order to apply it we need the following condition to hold, which guarantees that condition (ii) in the proposition holds:

(C") any $2$-ball $B$ in $\Delta'$ made of at most $4$ triangles with a common vertex has $f_{\Delta'_{\le 1}}$ injective on the vertices of $B$.

(Indeed,
the contraction of $v$ as in Lemma~\ref{lem:Schipper},
replaces its star by such ball $B$.)
For (C") to hold for $f_{\Delta_{\le 1}}(\Delta_{\le 1})$, $f_{\Delta_{\le 1}}(v)$ needs to avoid at most $48=6+6+12+24$ values in $A$.

Now, by Propositions~\ref{prop:WhiteleylikeContraction} and~\ref{prop:key} for $b=6$ we finish the proof as before:
among any $\binom{9}{3}=84$ points in $A$, there will be a point $a$ such that $f_{\Delta_{\le 1}}(v)=a$ makes the rigidity matrix of $\Delta_{\le 1}$ full rank.
Thus, by avoiding at most $48+83=131$ values in $A$, both (R) and (C") are guaranteed, given $n>n_0(M)$.
Thus, $c_3(F(g))\le \max(n_0(M),132)$.  $\square$

Note that when the (orientable or non-orientable) genus of $M$ tends to infinity, so does $n_0(M)$ (see Remark~\ref{rem:n_0(g)}). We do not know if $\{c_3(F(g))\}_g$ is bounded.

\section{Concluding remarks}\label{sec:conclude}
Regarding simplicial spheres, does a higher dimensional analog of Theorem~\ref{thm:2-spheres} hold? Namely,
\begin{problem}
Let $F(S^{d-1})$ be the family of $1$-skeleta of triangulations of the $(d-1)$-sphere. For $d>3$, is $F(S^{d-1})$ $d$-rigid with $c$-locations for some finite $c$? Namely, is $c_d(F(S^{d-1}))$ finite?
\end{problem}



Regarding surfaces, does a uniform bound in Theorem~\ref{thm:S_g} hold? Namely, for $F(S)$, the family of all graphs of compact connected surfaces, we ask:
\begin{problem}
Is $c_3(F(S))$ finite?
\end{problem}
We remark that for the larger family of Fogelsanger's minimal cycle complexes~\cite{Fogelsanger}, and even for the intermediate family which still contains $F(S)$, of complexes minimal with respect to containment among those supported by homology $2$-cycles, rigidity with few locations fails:
\begin{theorem}
There is no finite set of locations in $\R^3$ such that every minimal $2$-cycle can be realized with vertices in these locations in an infinitesimally rigid way.
\end{theorem}

\begin{proof}
The proof follows the idea in the first proof of Theorem~\ref{thm:no_c(F_d)}.

Indeed, assume that $c$ locations are enough to guarantee that every minimal $2$-cycle can be realized  in an infinitesimally rigid way. Consider a minimal $2$-cycle $\mu$ that requires $c$ locations.

Consider secondly the boundary of a tetrahedron $\Delta$, and mark the four vertices by $0,1,2$ and $3$. We subdivide the triangle $\{123\}$ in some way that
introduces exactly 3 new vertices, all in its interior, which form
a triangle $\Gamma$ (in the interior of the triangle $\{123\}$). Denote the resulting subdivision of $\Delta$ by $\Delta'$.

Totally order the triples of vertices in $\mu$. According to this order,
for every triple of vertices of $\mu$, attach a new copy of $\Delta'$ along the vertices $1$, $2$ and $3$ to the currently constructed complex $\mu'$ (starting with the original $\mu$), then remove $\Gamma$ and some triangle of $\mu'$, and connect both along a simplicial tube.

The resulted complex $\mu''$ is also a minimal $2$-cycle, and contains the $1$-skeleton of $\mu$, thus also requires $c$ locations; however it has more vertices. Repeating this process, we can assume $\mu$ has more than $2c$ vertices, so it has a triple of vertices $T$ occupying at most one location, and with the copy $w$ of $0\in \Delta$ corresponding to $T$, all $4$ vertices $T\cup\{w\}$ are contained in an affine plane $P$. Now, in $\mu''$, assign $w$ a nonzero velocity perpendicular to $P$, and zero velocity to all other vertices; this is a nontrivial infinitesimal motion on the $1$-skeleton of $\mu''$, a contradiction.
\end{proof}

Regarding Theorem~\ref{thm:no_c(F_d)}, while the phenomenon of rigidity with few locations does not hold for Laman graphs in the plane, Walter Whiteley asked us whether it does hold for the subfamily of planar Laman graphs, see~\cite[Thm.1]{PlanarLaman2005} for an equivalent characterization via pointed pseudo-triangulations.
We leave this question open.
\begin{problem}
Let $F$ be the family of planar Laman graphs. Is $c_2(F)$ finite?
\end{problem}

Given that rigidity with few locations holds, it is interesting to find optimal bounds. For $F(S^2)$, note that the vertices of the octahedron must occupy 6 different locations in $\R^3$ in any infinitesimally rigid embedding, thus $c_3(F(S^2))\ge 6$.
\begin{problem}\label{prob:PlanarLaman}
For every surface $M$, what is the value $c_3(F(M))$? At least, find improved bounds.
\end{problem}

In this note we considered infinitesimal rigidity with $c$-locations, for $c$ a constant. More generally, for a family $F$ of generically $d$-rigid graphs, let $F(n)$ be the subfamily of graphs in $F$ with at most $n$ vertices, and let $c_{d,F}(n)$ be the minimum $c$ such that $F(n)$ is $d$-rigid with $c$-locations. One can study the growth of the function $c_{d,F}(n)$.

\begin{problem}
For $F_3$, the family of all generically $3$-rigid graphs, what is the asymptotic growth of $c_{3,F_3}(n)$? Is it sublinear in $n$?
\end{problem}
Let us remark that for $G$ a minimally $d$-rigid graph, the chromatic number $\chi(G)$ of $G$ is a lower bound for $c_{d,F_d}(n)$, as infinitesimal $d$-rigidity forces the vertices of any edge to occupy two different locations. However, as any induced subgraph $G'=(V',E')$ of $G$ supports no nontrivial stress, $G'$ must satisfy $|E'|\leq d|V'|$, hence $G$ is $(2d-1)$-degenerate, so $\chi(G)\leq 2d$, and we get no growth with $n$ in the lower bound on $c_{d,F_d}(n)$ by using the chromatic number.

\section*{Acknowledgments}
We thank Isabella Novik and Orit Raz for helpful discussions, and Bob Connelly, Louis Theran and Walter Whiteley for helpful comments and references.
We thank the referee for valuable advice, resulted in greatly improved presentation.
K.A. was supported by ERC StG 716424 - CASe and ISF Grant 1050/16.
E.N. was partially supported by ISF grant 1695/15, by ISF-NRF joint grant 2528/16 and by ISF-BSF joint grant 2016288.

\newcommand{\etalchar}[1]{$^{#1}$}
\providecommand{\bysame}{\leavevmode\hbox to3em{\hrulefill}\thinspace}
\providecommand{\MR}{\relax\ifhmode\unskip\space\fi MR }
\providecommand{\MRhref}[2]{%
  \href{http://www.ams.org/mathscinet-getitem?mr=#1}{#2}
}
\providecommand{\href}[2]{#2}

\end{document}